\documentclass[12pt]{amsart}

\usepackage[margin=1in,letterpaper,portrait]{geometry}
\usepackage{hyperref}
\usepackage{mathtools}
\usepackage{stmaryrd}
\usepackage{xstring}
\usepackage{tikz}
\usepackage{xcolor}

\usepackage{enumitem}

\usepackage{cite}

\usetikzlibrary{patterns}

\usetikzlibrary{arrows,shapes,automata,backgrounds,decorations,petri,positioning}
\tikzstyle{edge} = [fill,opacity=.5,fill opacity=.5,line cap=round, line join=round, line width=50pt]

\theoremstyle{definition}
\newtheorem{proposition}{Proposition}[section]
\newtheorem{lemma}[proposition]{Lemma}

\newtheorem{theorem}[proposition]{Theorem}

\newtheorem{definition}[proposition]{Definition}
\newtheorem{example}[proposition]{Example}
\newtheorem{corollary}[proposition]{Corollary}

\newcommand{\rw}[1]{[{#1}]} 

\newcommand{\class}[1]{R_\bullet({#1})} 

\newcommand{\pinn}{\textsf{p}}
\newcommand{\vale}{\textsf{v}}
\newcommand{\pv}{\textsf{pv}}
\newcommand{\rev}[1]{{#1}^r}
\newcommand{\s}{\sigma}

\title{One-element commutation classes}

\author[Bridget Eileen Tenner]{Bridget Eileen Tenner$^*$}
\address{Department of Mathematical Sciences, DePaul University, Chicago, IL, USA}
\email{bridget@math.depaul.edu}
\thanks{$^*$Research partially supported by NSF Grant DMS-2054436.}


\makeatletter
\@namedef{subjclassname@2020}{%
  \textup{2020} Mathematics Subject Classification}
\makeatother

\begin{document}

\begin{abstract}
For any permutation $w$, we characterize the reduced words of $w$ that are their own commutation class. When $w$ is the long element $n(n-1)\cdots 321$ and $n \ge 4$, there are exactly four such words.
\end{abstract}

\maketitle

\section{Introduction}

Permutations in $\mathfrak{S}_n$ can be written as products of the simple reflections $\{\s_i : i \in [1,n-1]\}$, where $\s_i$ swaps $i$ and $i+1$. The products with minimal length are the \emph{reduced decompositions} of a permutation $w$, that length is the \emph{length} $\ell(w)$ of the permutation, and recording the string of subscripts in such a product gives a \emph{reduced word} of $w$. The set of reduced words of a permutation $w$ is $R(w)$, and it is known \cite{matsumoto, tits} that all elements of $R(w)$ are related by sequences of two local moves:
\begin{align*}
\text{commutation move: } & ij \leftrightarrow ji \text{ when } |i-j| > 1, \text{ and}\\
\text{braid move: } & i(i+1)i \leftrightarrow (i+1)i(i+1).
\end{align*}

Much work has been done to understand the set $R(w)$ and its equivalence classes, primarily as defined by the commutation relation (see, for example, \cite{bedard, bergeron, fishel, gutierres, jonsson welker}). More specifically, two reduced words are \emph{commutation equivalent}, denoted $\sim$, if they differ by a sequence of commutation moves. The \emph{commutation classes} of a permutation $w$ are the quotient of $R(w)$ by the equivalence relation $\sim$.

\begin{definition}
For any permutation $w$, let $\class{w} \subseteq R(w)$ be the reduced words of $w$ that are their own commutation classes.
\end{definition}

By definition, elements of $\class{w}$ are the reduced words of $w$ that support no commutation moves, meaning that they are precisely the reduced words of $w$ in which every pair of consecutive positions contains consecutive values. That being said, we can, in fact, describe elements of $\class{w}$ more precisely.

In this note, we characterize elements of $\class{w}$ for all $w$. This includes the special case $w = n(n-1)\cdots 321$, answering a question recently asked on \mbox{MathOverflow~\cite{mathoverflow}.}

\section{Preliminary results}

We first define several useful symmetries on strings. Throughout this discussion, we will distinguish the one-line notation of a permutation from its reduced words by writing reduced words using brackets; for example, $\rw{123} \in R(2341)$. 

\begin{definition}
Fix a string $s$ on the alphabet $\{1,2,\ldots, n-1\}$. The \emph{reverse} of $s$ is $\rev{s}$, obtained by writing $s$ in reverse order. The \emph{complement} of $s$ is $\overline{s}$, obtained by replacing each letter $i$ by $n-i$. These operations commute, and the \emph{reverse complement} of $s$ is $\rev{\overline{s}}$. The \emph{symmetries} of a reduced word $\rw{s}$ are the reduced words $\{\rw{s}, \rw{\rev{s}}, \rw{\overline{s}}, \rw{\rev{\overline{s}}}\}$. We can analogously define the \emph{symmetries} of any product of simple reflections.
\end{definition}

There is a natural visualization of these symmetries if we think of ``graphing'' a string $s = s_1 s_2 \cdots s_t$ by plotting the points $\{(i,s_i)\}$. Reversing reflects the graph across the vertical line at $(t+1)/2$, complementing reflects the graph across the horizontal line at $n/2$, and the reverse complement rotates the graph $180^\circ$ around the point $((t+1)/2,n/2)$.

A string $s$ represents a reduced word $\rw{s}$ if and only if all (any) of $s$'s symmetries also represent reduced words. That being said, while the symmetries of a reduced word are also reduced words, they might not be reduced words for the same permutation. For example, if $\rw{s} \in R(w)$, then $\rw{\rev{s}} \in R(w^{-1})$.

In the following lemma, and subsequently, we will say that a product $\s_{i_1}\cdots \s_{i_r}$ (likewise, the word representing the reflections' subscripts) is \emph{not reduced} if $\ell(\s_{i_1}\cdots \s_{i_r}) < r$. 

\begin{lemma}\label{lem:zig zag isn't reduced}
For any $i < j$, the product 
$$\s_i \s_{i+1} \cdots \s_{j-1} \s_j \s_{j-1} \cdots \s_{i+1}\s_i \s_{i+1} \cdots \s_{j-1}\s_j$$
is not reduced. The same is true for its symmetries. 
\end{lemma}

\begin{proof}
The given product has $3j-3i + 1$ terms. However, it defines the permutation
$$1 \cdots (i-1) {\color{red} (j+1) (i+2) (i+3) \cdots j i (i+1)} (j+2) (j+3) \cdots,$$
where everything not in red is fixed, and this permutation has length $3j - 3i - 1$.
\end{proof}

From this we conclude an important property of any element of $\class{w}$, which essentially combines Lemma~\ref{lem:zig zag isn't reduced} with a discrete version of the Intermediate Value Theorem.

\begin{corollary}\label{cor:increasing run}
Let $x < y$ be two letters appearing in a reduced word $\rw{s} \in \class{w}$. Then every consecutive subsequence of $\rw{s}$ of maximal length that occurs between the nearest instances of $x$ and $y$ is a sequence of consecutive values from $x$ to $y$.
\end{corollary}

\begin{proof}
Consecutive positions in $\rw{s}$ must have consecutive values. Thus, if the desired property did not hold, then, without loss of generality, we would have $x < i < j < y$ with
$$\rw{s} = \rw{\cdots x \cdots {\color{red} i(i+1)\cdots (j-1)j(j-1)\cdots(i+1)i(i+1)\cdots (j-1)j }\cdots y \cdots}.$$
But Lemma~\ref{lem:zig zag isn't reduced} means the red substring is not reduced, contradicting $\rw{s} \in R(w)$.
\end{proof}

We will use terminology from other work on permutations, introduced in \cite{dnpt} and appearing in several subsequent papers, but we are slightly more generous with the definitions here. In particular, we allow endpoints to qualify.

\begin{definition}
Let $s$ be a string of real numbers. The \emph{endpoints} of $s$ are its leftmost and rightmost values. A \emph{pinnacle} of $s$ is a value that is larger than its immediate neighbor(s), and a \emph{vale} is a value that is smaller than its immediate neighbor(s). Write $\pinn(s)$ for the substring of pinnacles of $s$, and $\vale(s)$ for the substring of vales. The substring of pinnacles and vales will be written $\pv(s)$.
\end{definition}

The pinnacles and vales necessarily alternate in $\pv(s)$. Moreover, the endpoints of $s$ always appear in $\pv(s)$. 

\begin{example}
In the string $s = {\color{red} 6}543{\color{blue} 2}3{\color{red} 4}32{\color{blue} 1}2345{\color{red} 6}5{\color{blue} 4}{\color{red} 5}4{\color{blue} 3}4{\color{red} 5}$, the pinnacles have been marked in red, the vales have been marked in blue, and $\pv(s)  = {\color{red} 6}{\color{blue} 2}{\color{red} 4}{\color{blue} 1}{\color{red} 6}{\color{blue} 4}{\color{red} 5}{\color{blue} 3}{\color{red} 5}$.
\end{example}

We will want to recognize when strings have a property similar to unimodality.

\begin{definition} 
Let $s = s_1\cdots s_t$ be a string of real numbers. If there exist $i$ and $j$ such that $1 \le i \le j \le i+1 \le t$ and
$$s_1 < \cdots < s_i = s_j > \cdots > s_t,$$
then $s$ is a \emph{wedge}. If
$$s_1 > \cdots > s_i = s_j < \cdots < s_t,$$
then $s$ is a \emph{vee}. 
If $i = j$, then that wedge or vee is \emph{strict}.
\end{definition}

\section{The main result}

\begin{theorem}\label{thm:characterizing}
For any $w$, if $\rw{s} \in \class{w}$ then
\begin{enumerate}[label=(\alph*)]
\item $\pinn(s)$ is a wedge,
\item $\vale(s)$ is a vee, 
\item at least one of $\pinn(s)$ or $\vale(s)$ is strict, 
\item the minimum and maximum values of $\pv(s)$ appear consecutively, and
\item if $\pinn(s)$ (or $\vale(s)$) has more than one $x$, then one of those $x$s is an endpoint of $s$.
\end{enumerate}
\end{theorem}

The proof of the theorem will involve property about certain reduced words.

\begin{lemma}\label{lem:critical repeated pinnacles}
Suppose that $\rw{s} \in \class{w}$.
\begin{itemize}
\item Suppose that $x$ appears twice in $\pinn(s)$, with one of those appearances being the right (resp., left) endpoint of $s$ and the other not being the left (resp., right) endpoint of $s$. Then $\ell(w\s_{x-1}) < \ell(w)$ (resp., $\ell(\s_{x-1}w) < \ell(w)$). 
\item Suppose that $x$ appears twice in $\vale(s)$, with one of those appearances being the right (resp., left) endpoint of $s$ and the other not being the left (resp., right) endpoint of $s$. Then $\ell(w\s_{x+1}) < \ell(w)$ (resp., $\ell(\s_{x+1}w) < \ell(w)$). 
\end{itemize}
\end{lemma}

\begin{proof}
We will prove one of these results, and the others can be proved by analogous arguments.

Suppose that $x$ appears twice in $\pinn(s)$, with one of those appearances being the right endpoint of $s$ and the other not being the left endpoint of $s$. Then $\rw{s}$ has a form like
$$\rw{\underbrace{\cdots}_{<x-1} (x-1)x(x-1) \underbrace{\cdots}_{<x-1} (x-1)x \underbrace{\cdots}_{>x} x(x-1) \underbrace{\cdots}_{<x-1} (x-1)x \underbrace{\cdots}_{>x} x(x-1) \underbrace{\cdots}_{<x-1} (x-1)x}.$$
From this we can do a sequence of braid and commutation moves to produce additional elements of $R(w)$, as demonstrated below, where the changes resulting from braids are indicated in red and from commutations in blue:
\begin{eqnarray*}
&\rw{\underbrace{\cdots}_{<x-1} {\color{red}x(x-1)x} \underbrace{\cdots}_{<x-1} (x-1)x \underbrace{\cdots}_{>x} x(x-1) \underbrace{\cdots}_{<x-1} (x-1)x \underbrace{\cdots}_{>x} x(x-1) \underbrace{\cdots}_{<x-1} (x-1)x},&\\
&\rw{\underbrace{\cdots}_{<x-1} x(x-1) {\color{blue}\underbrace{\cdots}_{<x-1}} {\color{blue}x}(x-1)x \underbrace{\cdots}_{>x} x(x-1) \underbrace{\cdots}_{<x-1} (x-1)x \underbrace{\cdots}_{>x} x(x-1) \underbrace{\cdots}_{<x-1} (x-1)x},&\\
&\rw{\underbrace{\cdots}_{<x-1} x(x-1) \underbrace{\cdots}_{<x-1} {\color{red}(x-1)x(x-1)} \underbrace{\cdots}_{>x} x(x-1) \underbrace{\cdots}_{<x-1} (x-1)x \underbrace{\cdots}_{>x} x(x-1) \underbrace{\cdots}_{<x-1} (x-1)x},&\\
&\vdots&\\
&\rw{\underbrace{\cdots}_{<x-1} x(x-1) \underbrace{\cdots}_{<x-1} (x-1)x \underbrace{\cdots}_{>x} x(x-1) \underbrace{\cdots}_{<x-1} (x-1)x \underbrace{\cdots}_{>x} x(x-1) \underbrace{\cdots}_{<x-1} {\color{red}(x-1)x(x-1)}}.&
\end{eqnarray*}
Thus $\ell(w\s_{x-1}) < \ell(w)$.
\end{proof}

We can now prove the main result. 

\begin{proof}[Proof of Theorem~\ref{thm:characterizing}]
We prove the theorem by induction on the length of the permutation $w$, and the result is trivial to check in small cases. Suppose that $\rw{s} \in \class{w}$, that $\ell(w) = t > 3$, and that the result holds for all permutations of length less than $t$. Because $\rw{s}$ is a reduced word, we know that consecutive positions of $\rw{s}$ always containing consecutive values is equivalent to the condition $\rw{s} \in \class{w}$.

Consider $\rw{s} \in \class{w}$, where $s = s_1 \cdots s_t$. Define the string $s'$ and the permutation $w'$ so that $s' = s_1 \cdots s_{t-1}$ and $\rw{s'} \in R(w')$. That $\rw{s'}$ is a reduced word follows from the fact that $\rw{s}$ is reduced. Moreover, $\rw{s'} \in \class{w'}$ because $\rw{s}$ supports no commutation moves, and necessarily $\rw{s'}$ supports none either. Thus the inductive hypothesis applies to $\rw{s'} \in R(w')$. Without loss of generality, assume that $s_{t-1}$ is a pinnacle of $s'$ (the case of a vale can be argued analogously.) There are now two cases, based on the relative sizes of $s_{t-1}$ and $s_t$.

\vspace{.1in}

\noindent \framebox{Case 1:} $s_t = s_{t-1} - 1$. Thus $\pinn(s) = \pinn(s')$ (so property (a) holds for $s$) and $\vale(s) = \vale(s') s_t$.

\vspace{.1in}
Let $x < s_{t-1}$ be the vale preceding $s_{t-1}$; that is, $\pv(s') = \cdots xs_{t-1}$. 

If $x = s_t$, then $x$ appears immediately to the left of $s_{t-1}$ in $s$. Then $x$ must be the left endpoint of $s$, in order to avoid violating Lemma~\ref{lem:zig zag isn't reduced}. From here it is obvious that $s$ has the desired properties.

Now consider $x < s_t$. This gives properties (b) and (c). Property (d) follows because we have not altered the pinnacle sequence and we have only added a non-minimal value to the vale sequence. Regarding property (e), it is okay if $s_t$ repeats a value of $\vale(s')$ because $s_t$ is an endpoint of $s$. Thus the only potential concern is that the value $s_{t-1}$ might have appeared twice in $\pinn(s')$, and not as the left endpoint of $s'$. But then, by Lemma~\ref{lem:critical repeated pinnacles}, the word $\rw{s}$ would not be reduced, which would be a contradiction. 

\vspace{.1in}

\noindent \framebox{Case 2:} $s_t = s_{t-1} + 1$. Thus $\vale(s) = \vale(s')$ (so property (b) holds for $s$) and $\pinn(s)$ is obtained from $\pinn(s')$ by replacing its last letter $s_{t-1}$ by $s_t$.

\vspace{.1in}
If $\pinn(s') = s_{t-1}$, then the result is immediate. Otherwise, let $y$ be the preceding pinnacle. If this is the left endpoint of $s$, then the result is straightforward to check. So let us suppose that $\pv(s') = \cdots x_1yx_2s_{t-1}$.

If $y \le s_{t-1}$, then $s_{t-1}$ was a maximum value in $\pinn(s')$ and thus $s_t$ is the unique maximum value in $\pinn(s)$ and the requirements for $s$ follow from the inductive hypothesis on $s'$.

If $y = s_{t-1} + 1 = s_t$, then we need $x_1 > x_2$ in order for Lemma~\ref{lem:zig zag isn't reduced} not to contradict the fact that $\rw{s} \in R(w)$. Then the wedge properties of $s'$ says that $x_2$ is the unique minimum of $\vale(s')$, and requirement (c) on $s'$ means that $y$ is the unique maximum in $\pinn(s')$. Thus $\rw{s}$ has the desired properties. 

Finally, suppose $y > s_{t-1} + 1$. Then $\pinn(s)$ inherits wedge properties and strictness from $\pinn(s')$, so property (a) holds. And because $\vale(s) = \vale(s')$, we also have properties (c) and (d). As for property (e), it is okay if $s_t$ repeats a value of $\vale(s')$ because $s_t$ is an endpoint of $s$. The value $s_{t-1}$ appears strictly fewer times in $\pinn(s)$ than it had done in $\pinn(s')$, so there is no concern about property (e) regarding this value.

\vspace{.1in}

Therefore the reduced word $\rw{s}$ has the desired properties.
\end{proof}

\section{The long element}

We conclude this note by considering the special case of the long element $w_0 = n(n-1)\cdots 321 \in \mathfrak{S}_n$. When $n < 4$, it is trivial to see that $\class{w_0} = R(w_0)$. On the other hand, when $n \ge 4$, the set $\class{w}$ has exactly four elements.

\begin{corollary}\label{cor:long element}
The set $\class{n(n-1)\cdots 321}$ consists exactly of the symmetries of
\begin{equation}\label{eqn:the special string}
\rw{123\cdots (n-1)\cdots 321 \ 234\cdots (n-2) \cdots 432 \ 345 \cdots (n-3) \cdots 543 \ \cdots}.
\end{equation}
\end{corollary}

\begin{proof}
Let $\rw{s}$ be the string given in \eqref{eqn:the special string}. First note that when $n \ge 4$, the symmetries of $\rw{s}$ are all distinct. Also, when $i < j$, we have that $\rw{i(i+1)\cdots (j-1)j(j-1)\cdots (i+1)i}$ is a reduced word for the permutation swapping $i$ and $j+1$ and fixing all other values. Thus $\rw{s} \in R(w_0)$, and the symmetries of $\rw{s}$ are also reduced words for $w_0$.

It remains to show that there are no other elements of $\class{w_0}$. For the remainder of the proof, suppose that $n \ge 4$, and assume the result holds for $\class{m\cdots 321}$ when $m < n$.

Fix $\rw{t} \in \class{w_0}$. Suppose, for the purpose of obtaining a contradiction, that $\rw{t}$ contains exactly one $1$ and exactly one $n-1$. Without loss of generality, suppose that the $1$ appears to the left of the $n-1$. The simple reflection $\s_1$ moves $1$ out of position in the permutation, and later multiplying on the right by $\s_{n-1}$ moves $n$ out of position. However, there will not subsequently be a way to move $n$ into the first position of the permutation, which is a contradiction. Thus exactly one of $\pinn(t)$ and $\vale(t)$ is strict. In particular, $\rw{t}$ must have either
$$\rw{a}:= \rw{123\cdots(n-1)\cdots321} \hspace{.25in} \text{or} \hspace{.25in} \rw{\overline{a}}$$
as a prefix or a suffix. In other words, we are writing $w_0$ as a product of $n23\cdots (n-1)1$ and the permutation $v:= 1(n-1)(n-2)\cdots 432n$. From the inductive hypothesis, and a small reindexing, we know that elements of $\class{v}$ are exactly the symmetries of
$$\rw{b} := \rw{234\cdots (n-2) \cdots 432 \ 345 \cdots (n-3) \cdots 543 \ 456 \cdots (n-4) \cdots 654 \ \cdots}.$$
Consecutive positions of $t$ have consecutive values, so there are only four possibilities for $\rw{t}$:
$$\rw{ab}, \ \rw{ba} , \ \rw{\overline{ab}} , \ \text{and} \ \rw{\overline{ba}}.$$
Since $\rev{a} = a$ and $\rev{b} = b$, these are exactly the symmetries of the string given in \eqref{eqn:the special string}.
\end{proof}

Finally, we note that the long element is not the only permutation with $|\class{w}| = 4$. For example,
\begin{align*}
\class{7265413} = \big\{ & 4345654321234543, 3456543212345434\\
& 5434565432123454, 4543456543212345\big\}.
\end{align*}

\section*{Acknowledgements}

I am grateful to Richard Stanley for drawing my attention to the MathOverflow post \cite{mathoverflow} and for inspired terminology. I appreciate the comments and advice of an anonymous referee.

\end{document}